\documentclass[psamsfonts,12pt]{amsart}

\usepackage{enumerate}
\usepackage{amsmath}
\usepackage{color}
\usepackage{cite}

\newcommand{\R}{\mathbb{R}}
\newcommand{\N}{\mathbb{N}}

\newcommand{\di}{\mathop{\mbox{\rm{diad}}}\nolimits}
\newcommand{\cD}{\mathcal{D}}

\newtheorem{thm}{Theorem}
\newtheorem*{thm*}{Theorem}

\newtheorem{cor}[thm]{Corollary}

\newtheorem{cl}[thm]{Claim}

\newtheorem{obvs}{Observations}
\newtheorem{que}{Question}
\newtheorem{lem}[thm]{Lemma}
\theoremstyle{definition}
\newtheorem{Def}[thm]{Definition}

\newtheorem{open}{Open problem}

\begin{document}


\title[Dichotomy for strictly increasing bisymmetric maps]{A dichotomy result for strictly increasing bisymmetric maps}
\author{P\'al Burai, Gergely Kiss and Patricia Szokol}
\address{P\'al Burai, \newline Budapest University of Technology and Economics, \newline 1111 Budapest,
Műegyetem rkp. 3., HUNGARY}
\email{buraipl@math.bme.hu}
\address{Gergely Kiss, \newline Alfr\'ed R\'enyi Institute of Mathematics,\newline  1053 Budapest, Re\'altanoda street 13-15, HUNGARY}
\email{kiss.gergely@renyi.hu}
\address{Patricia Szokol, \newline University of Debrecen,\newline
Faculty of Informatics, University of Debrecen,
MTA-DE Research Group “Equations, Functions and Curves”, \newline
4028, Debrecen, 26 Kassai road, Hungary}
\email{szokol.patricia@inf.unideb.hu}
\maketitle


\begin{abstract}
In this paper we show some remarkable consequences of the method which proves that every bisymmetric, symmetric, reflexive, strictly monotonic binary map on a proper interval is continuous, in particular it is a quasi-arithmetic mean. Now we demonstrate that this result can be refined in the way that the symmetry condition can be weakened by assuming symmetry only for a pair of distinct points of an interval. 
\end{abstract}


\section{Introduction}

The role of bisymmetry in the characterization of binary quasi-a\-rith\-metic means goes back to the research of J\'anos Acz\'el (see \cite{Aczel1948}). This led him to a new approach, which is basically different from the earlier multivariate characterization of quasi-arithmetic means by Kolmogoroff \cite{Kolmogoroff1930}, Nagumo \cite{Nagumo1930} and de Finetti \cite{deFinetti1931}. Since that time quasi-arithmetic means became central objects in theory of functional equations, especially in the theory of means (see e.g. \cite{Burai2013c, Daroczy2013,Duc2020,Glazowska2020,Kiss2018,LPZ2020,Nagy2019,Pales2011,Pasteczka2020}, 
in particular \cite{Daroczy2002} and \cite{Jar2018} and the references therein).

In the proof of Aczél's characterization (Theorem \ref{Aczel1}, for details see \cite[Theorem 1 on p. 287]{Aczel1989}) the assumption of continuity was essentially used. It seemed that continuity cannot be omitted from the conditions of Theorem \ref{Aczel1} until quite recently the authors showed that the characterization of two-variable quasi-arithmetic means is possible without the assumption of continuity (Theorem \ref{T:bisymmetryimpliescontinuity}, for details see \cite[Theorem 8]{BKSZ2021}). It was proved that every bisymmetric, symmetric, reflexive, strictly monotonic binary mapping $F$ on a proper interval $I$ is continuous, in particular it is a quasi-arithmetic mean.

In this paper we show a nontrivial consequence of Theorem \ref{T:bisymmetryimpliescontinuity}. 
We prove a dichotomy result of bisymmetric, reflexive, strictly monotonic operations on an interval concerning symmetry (Corollary \ref{cor1}). Namely, such functions are either symmetric everywhere or nowhere symmetric.
In this sense this paper can be seen as the next step toward the characterization of bisymmetric, partially strictly monotone operations (see Open Problem \ref{op1}).


The remaining part of this paper organized as follows.
In Section \ref{S:preliminary} we introduce the basic definitions and preliminary results. Section \ref{s3} is devoted to our main result (Theorem \ref{T:bisymmetryimpliescontinuity2}) and its consequences. Here we show some illustrative examples for the strictness of our main result.
The proof of Theorem \ref{T:bisymmetryimpliescontinuity2} is a quite lengthy and technical one. Therefore, we introduce at first the needed concepts and prove some important lemmata in Section \ref{s31}, while Section \ref{s32} is devoted to the proof of Theorem \ref{T:bisymmetryimpliescontinuity2}. We finish this short note with some concluding remarks.

\section{Notations}\label{S:preliminary}

We keep the following notations throughout the text. Let $I\subseteq\R$ be a  proper interval (i.e. the interior of $I$ is nonempty) 
and $F\colon I\times I\to\R$ be a map.

Then $F$ is said to be
\begin{enumerate}[(i)]
\item {\it reflexive}, if $F(x,x)=x$ for every $x\in I$;
\item {\it partially strictly increasing}, if the functions $$x\mapsto F(x,y_0),\quad  y\mapsto F(x_0,y)$$ are strictly increasing for every fixed $x_0\in I$ and $y_0\in I$. One can define partially strictly monotone, partially monotone, partially increasing functions similarly;
\item {\it symmetric}, if $F(x,y)=F(y,x)$ for every $x,y\in I$;
\item {\it bisymmetric}, if
\begin{equation*}\label{E:bisymmetry}
 F\big(F(x,y),F(u,v)\big)=F\big(F(x,u),F(y,v)\big)
\end{equation*}
for every $x,y,u,v\in I$;
\item \emph{left / right cancellative}, if $F(x,a)=F(y,a)$ / $F(a,x)=F(a,y)$  implies $x=y$ for every $x,y,a\in I$. If $F$ is both left and right cancellative, then we simple say that $F$ is {\it cancellative}.
\item \emph{mean}, if \begin{equation*}
\min\{x,y\}\leq F(x,y)\leq\max\{x,y\}
\end{equation*} for every $x,y\in I$. $F$ is a {\it strict mean} if the previous inequalities are strict whenever $x\ne y$.
\end{enumerate}


\begin{obvs}\label{o1}
If a map $F:I^2\to I$ is partially strictly increasing, then it is cancellative.
\end{obvs}

The following fundamental result is due to Acz\'el \cite{Aczel1948} (see also \cite[Theorem 1 on p. 287]{Aczel1989}).

\begin{thm}\label{Aczel1}
A function $F:I^2\to I$ is continuous, reflexive, partially strictly monotonic, symmetric and bisymmetric  mapping if and only if there is a continuous, strictly increasing function $f$ 
that satisfies
\begin{equation}\label{eqa1}
    F(x,y)=f \left(\frac{f^{-1}(x)+f^{-1}(y)}{2}\right),\qquad x,y\in I.
\end{equation}
\end{thm}

A function $F$ which satisfies \eqref{eqa1} is called a {\it quasi-arithmetic mean}.  In other words, a quasi-arithmetic mean is a conjugate of the arithmetic mean by a continuous bijection $f$.

In  \cite{BKSZ2021} the authors showed that the assumption of continuity for $F$ in Theorem \ref{Aczel1} can be omitted insomuch as it is the consequence of the remaining conditions.

\begin{thm}\label{T:bisymmetryimpliescontinuity}
A function $F\colon I^2\to I$ is reflexive, partially strictly increasing, symmetric and bisymmetric mapping if and only if there is a continuous, strictly monotonic function $f$ such that
\begin{equation}\label{Eq_foalak}
F(x,y)=f\left(\frac{f^{-1}(x)+f^{-1}(y)}{2}\right),\qquad x,y\in I.
\end{equation}

In particular every reflexive, partially strictly increasing, symmetric and bisymmetric binary mapping defined on $I$ is continuous.
\end{thm}

\section{Dichotomy result on the symmetry of bisymmetric, strictly monotone, reflexive binary functions}\label{s3}

We prove that a reflexive, bisymmetric, partially strictly increasing map is either totally symmetric or totally non-symmetric on the whole domain.

The main result of this section runs as follows:

\begin{thm}\label{T:bisymmetryimpliescontinuity2}
Let us assume that $I$ is a proper interval and $F\colon I^2\to I$ is a reflexive, partially strictly increasing and bisymmetric map. Suppose that there is an $\alpha,\beta\in I$ ($\alpha\ne \beta$) such that $F(\alpha,\beta)=F(\beta,\alpha)$.  Then $F$ is symmetric on $I$ and continuous, i.e., $F$ is a quasi-arithmetic mean. 
\end{thm}
As an immediate consequence of Theorem \ref{T:bisymmetryimpliescontinuity2} we get the following dichotomy result.
\begin{cor}\label{cor1}
Let $I$ be a proper interval, then
every bisymmetric, partially strictly increasing, reflexive, binary function $F\colon I^2\to I$ is
either symmetric everywhere on $I$ or nowhere symmetric on $I$.
\end{cor}
We illustrate the strictness of our main results with some examples.

The map
\[
F\colon[0,1]^2\to[0,1],\quad F(x,y):=\begin{cases}
\frac{2xy}{x+y}&\mbox{if } x\in[0,1], \ \mbox{and } y\in [0,\tfrac12[\\
\sqrt{xy}&\mbox{if } x\in [0,\tfrac12],\ \mbox{and } y\in [\tfrac12,1]\\
\frac{x+y}{2}&\mbox{otherwise}
\end{cases}
\]
is reflexive, partially strictly monotone increasing, not bisymmetric and it is neither symmetric nor non-symmetric for every elements of $[0,1]^2$.

\medskip
The map
\[
F\colon[0,1]^2\to[0,1],\quad F(x,y):=\begin{cases}
y&\mbox{if } x,y\in[\tfrac12,1]\\
\min\{x,y\}&\mbox{otherwise}
\end{cases}
\]
is reflexive, bisymmetric, partially monotone increasing but not strictly, and it is neither symmetric nor non-symmetric for every elements of $[0,1]^2$.

\medskip
Concerning the relaxation of reflexivity condition we can formulate the following open problem.

\begin{open} \label{op1}
Is it true or not that every bisymmetric, partially strictly increasing map is automatically continuous?
\end{open}

If the answer is affirmative, then the resulted map can be written in the following form  (see \cite[Exercise 2, p. 296]{Aczel1989}).
\[
F(x,y)=k^{-1}(ak(x)+bk(y)+c),
\]
where $k$ is an invertible, continuous function, and $a,b,c$ are arbitrary real constants such that $ab\not=0$. In this case $F$ is either symmetric or non-symmetric everywhere. It is reflexive only if $c=0$ and $a+b=1$.

We could not find a map which is bisymmetric, partially strictly increasing, not reflexive and neither symmetric nor non-symmetric for every pair of $I^2$.
\subsection{Auxiliary results and needed concepts}\label{s31}

\phantom{nnn}

Let $(u,v,F)_n$ denote the set of all expressions that can be build as n-times compositions of $F$ by using $u$ and $v$.
For instance
\begin{align*}
(u,v,F)_0=&\{u,v\}\\
(u,v,F)_1=&\{F(u,u),F(u,v), F(v,u), F(v,v)\}\\
(u,v,F)_2=&\{F(F(u,u),u),F(F(u,v),u), F(F(v,u),u), F(F(v,v),u),\\ F(F(u,u),v)&,F(F(u,v),v), F(F(v,u),v), F(F(v,v),v), F(u,F(u,u)), \\ F(u,F(u,v))&, F(u,F(v,u)), F(v,F(v,v)),F(v,F(u,u)),F(v,F(u,v)), \\ F(v,F(v,u))&, F(v,F(v,v))\}.
\end{align*}

Moreover, let $(u,v,F)_{\infty}$ denote the set of all expressions that can be build as any number of compositions of $F$ by using $u$ and $v$. Formally, $$(u,v,F)_{\infty}=\bigcup_{n=1}^{\infty}(u,v,F)_{n}. $$

Reflexivity implies that $(u,v,F)_k\subset (u,v,F)_n$, if $k<n$. Hence, for the sake of convenience, we can introduce the notion of the length of expressions of $(u,v,F)_{\infty}$ as follows.
Let $x\in (u,v,F)_{\infty}$, such that
\[
\min_{k\in\N}\{x\in (u,v,F)_k\}=k_0.
\]
Then $k_0$ is called the length of $x$. Notation: $\mathcal{L}(x)=k_0$.

For example the length of $x = F(F(u,u),v)=F(u,v)$ is $1$.

We go on this subsection with the proof of some technical lemmata.

\begin{lem} If $F(u,v)=F(v,u)$, then F is symmetric for any pair $(t,s)$, where $t,s\in (u,v, F)_{\infty}$. \end{lem}

\begin{proof}
We prove it by induction with respect to the length of the elements of $(u,v, F)_{\infty}$. It is easy to check that $F$ is symmetric for any pair of elements $(u,v,F)_k$ if $k=0,1$. For instance $F$ is symmetric for $(u,F(u,v))$.
Indeed, applying the reflexivity and bisymmetry of $F$ and  $F(u,v)=F(v,u)$, we get
\begin{eqnarray*}
F(u,F(u,v))=F(F(u,u),F(v,u))=F(F(u,v),u).
\end{eqnarray*}

Similarly, $F$ is symmetric for  $\{(u,F(v,u)), (v,F(u,v)), (v,F(v,u)) \}$, and hence for any pair of elements of $(u,v,F)_1$. 

Now, we prove that $F$ is symmetric for any pair of elements of $(u,v,F)_k$, where $k\le n+1$, under the assumption that $F$ is symmetric for any pair of elements of $(u,v,F)_k$, where $k\le n$.
Let $x$ and $y$ be two elements of $(u,v, F)_{\infty}$ such that
$\mathcal{L}(x)=k$, $\mathcal{L}(y)=l$ where $k,l \le n+1$. Then, there exists $a,b,c,d \in (u,v, F)_{\infty}$, such that $\mathcal{L}(a), \mathcal{L}(b), \mathcal{L}(c), \mathcal{L}(d)\leq n$ and
\[
x=F(a,b), \qquad y=F(c,d).
\]
By the inductive hypothesis we get that $F$ is symmetric for each pair of the set $\{a,b,c,d\}$. Applying the bisymmetry of $F$ we obtain
\begin{eqnarray*}
F(x,y)&=&F(F(a,b),F(c,d))=F(F(a,c),F(b,d))\\
      &=&F(F(c,a),F(d,b))=F(F(c,d),F(a,b))=F(y,x).
\end{eqnarray*}
\end{proof}

\begin{lem}\label{l1}
Let $I$ be a proper interval, and $F:I^2\to I$ be a bisymmetric, partially strictly increasing and reflexive function. Suppose that
there are $u,v\in I,\ u<v$ such that $F(u,v)=F(v,u)$. Then there is an invertible, continuous function $f\colon[0,1]\to [u,v]$ such that $F$ can be written in the form
\begin{equation}\label{eqam}
F(x,y)=f\left(\frac{f^{-1}(x)+f^{-1}(y)}{2}\right),\qquad x,y\in [u,v].
\end{equation}
In particular, $F(s,t)=F(t,s)$ holds for all $s,t\in [u,v]$.
\end{lem}


\begin{proof}
The argument is similar to the proof\footnote{for the details see \cite[proof of Theorem 8 on page 479]{BKSZ2021}} of Theorem \ref{T:bisymmetryimpliescontinuity}. The main observation is that the proof 
use only the symmetry of $F$ on the images of dyadic numbers which is exactly the set $(u,v,F)_{\infty}$ in our case. For convenience we briefly sketch the crucial steps of the  proof.

\begin{itemize}
    \item Define $f\colon[0,1]\to[u,v]$ recursively on the set of dyadic numbers $\cD$ to the set $(u,v,F)_{\infty}$, so that $f(0)=u,f(1)=v, f(\frac{1}{2})=u\circ v$ and $f$ satisfies the identity \begin{equation}\label{identity_on_dyadics}
f\left(\frac{d_1+d_2}{2}\right)=F(f(d_1), f(d_2))
\end{equation}
for every $d_1,d_2\in\cD$. There can be proved that such an $f$ is well-defined and strictly increasing. In this argument we crucially use the fact that $F$ is symmetric on $(u,v,F)_{\infty}$. By its recursive definition, it is clear that $f(\cD)=(u,v,F)_{\infty}$. (See also Acz\'el and Dhombres \cite{Aczel1989} on the pages $287-290$.)
\item The closure of $f(\cD)$ has uncountably many two-sided accumulation points\footnote{A point $\alpha$ in a set $H$ is a {\it two-sided accumulation point} if for every
$\varepsilon>0$, we have
\[
]\alpha-\varepsilon,\alpha[~\cap~H\not=\emptyset\quad\mbox{ and }\quad ]\alpha,\alpha+\varepsilon[~\cap~H\not=\emptyset.
\].}.

\item If $f(\cD)$ is not dense in $[u,v]$, i.e., there are $X,Y\in [u,v]$ such that $]X,Y[~\cap ~ f(\cD)= \emptyset$, then one can show that for arbitrary two-sided accumulation points $s\not=t$ we have
\[
]F(X,s),F(Y,s)[~\cap~ ]F(X,t),F(Y,t)[ ~=\emptyset.
\]
Hence the cardinality of disjoint intervals as well as the cardinality of two-sided accumulation points is uncountable, which gives a contradiction. Thus, $f(\cD)$ has to be dense in $[u,v]$.
\item If $f(\cD)$ is dense in $[u,v]$, then $f$ can be defined strictly increasingly on $[0,1]$, so that $f$ is continuous and satisfies \eqref{eqam}.
\end{itemize}
\end{proof}





\begin{lem}\label{l:union}
Let $I_1, I_2\subseteq I$ be two intervals such that $F$ is symmetric on $I_1$ and $I_2$.  Then $F$ is symmetric on $F(I_1, I_2):=\{\ F(x_1,x_2)\ |\ x_1\in I_1,\ x_2\in I_2\ \}$.
Furthermore,  if $I_1\cap I_2\not=\emptyset$, then $F$ is symmetric on $I_1\cup I_2$.
\end{lem}

\begin{proof}
We have to show that $F(z_1, z_2)=F(z_2, z_1)$ for $z_1, z_2\in F(I_1, I_2)$.
Let $x_1,x_2\in I_1$ and $y_1,y_2\in I_2$ such that $F(x_1, y_1)=z_1$ and $F(x_2, y_2)=z_2$.
Then \begin{align*}
&F(z_1, z_2)=F(F(x_1, y_1), F(x_2, y_2))=F(F(x_1, x_2), F(y_1, y_2))=\\&F(F(x_2, x_1),F(y_2, y_1))=F(F(x_2,y_2),F(x_1, y_1))=F(z_2, z_1),
    \end{align*}
where in the second and fourth equalities we use bisymmetry and the third equality holds by the symmetry of $F$ on $I_1$ and $I_2$.

Now, let us assume that $z\in I_1\cap I_2$ and let $x\in I_1$, $y\in I_2$ be arbitrary. We have to show, that $F(x,y)=F(y,x)$.

Using $F(x,z)=F(z,x)$, $F(y,z)=F(z,y)$ and the bisymmetry of $F$, we get
\begin{eqnarray*}
F(F(x,y),F(z,z))&=&F(F(x,z),F(y,z))=\\F(F(z,x),F(y,z))&=&F(F(z,y),F(x,z))=\\
F(F(y,z),F(x,z))&=&F(F(y,x),F(z,z)).
\end{eqnarray*}
Since $F$ is partially strictly increasing, by Observation \ref{o1}, it is cancellative and hence $F(x,y)=F(y,x)$.
\end{proof}

Now, we are in the position to prove our main theorem.

\subsection{Proof of Theorem \ref{T:bisymmetryimpliescontinuity2}}\label{s32}

\phantom{nnn}

Let us assume first that $I$ is a proper compact interval.

Let $\sim$ be defined on $I$ such that for any $a,b\in I$ we have $a\sim b$ if and only if $F(a,b)=F(b,a)$. Then $\sim$ is an equivalence relation.
Indeed, $\sim$ is clearly reflexive and symmetric. Transitivity is a direct consequence of Lemma \ref{l:union}.

Lemma \ref{l1} guarantees that if two points are in the same equivalence class, then the interval between them belongs to the same class. Combining this fact with the transitivity of $\sim$, we can obtain that every equivalence class is an interval.

One can introduce an ordering $<$ between the equivalence classes of $\sim$ as follows.  
For two equivalence classes $I_1, I_2$ ($I_1\ne I_2$) we say that $I_1$ smaller than $I_2$ (denote it by $I_1<I_2$) if every element of $I_1$ is smaller than every element of $I_2$. As every equivalence class is an interval, this definition is meaningful and gives a natural total ordering on the equivalence classes of $F$ in $I$.

\textbf{Step 1:} {\it  Let $I_1,I_2\subseteq  I$ be two equivalence classes such that $I_1 < I_2$, then $F(I_1, I_3)< F(I_2, I_3)$ (resp. $F(I_3, I_1)< F(I_3, I_2)$) for every equivalence class $I_3$. In particular, if $I_1<I_2$, then $I_1<F(I_1, I_2)<I_2$.}

By Lemma \ref{l:union}, we get that $F$ is symmetric on $F(I_1,I_3)$ and on $F(I_2,I_3)$. If these sets are disjoint, then $F(I_1, I_3)< F(I_2, I_3)$  by partially strictly increasingness of $F$. Now, assume that there exists a common element of $F(I_1,I_3)$ and $F(I_2,I_3)$, i.e., there exist $x_1\in I_1$, $x_2\in I_2$ and $y_1,y_2 \in I_3$ such that $F(x_1,y_1)=F(x_2, y_2)$. Hence,
\[
F(F(x_1,y_1),F(x_2,y_2))=F(F(x_2,y_2),F(x_1,y_1)).
\]
By bisymmetry, the left hand-side is equal to $F(F(x_1,x_2),F(y_1,y_2))$. Concerning the right-hand-side, bisymmetry and the fact $y_1\sim y_2$ implies that
\[
F(F(x_1,x_2),F(y_1,y_2))=F(F(x_2,x_1),F(y_1,y_2)).
\]
Moreover, $F$ is partially strictly increasing and hence, by Observation \ref{o1}, it is cancellative. Consequently, $F(x_1,x_2)=F(x_2,x_1)$, which is a contradiction, since $x_1$ and $x_2$ are belonging to two different equivalence classes.

Similarly, we can get that $F(I_3, I_1)< F(I_3, I_2)$ for any $I_3$, if $I_1< I_2$.
In particular, the choice $I_1=I_3$ gives that $F(I_1, I_1)=I_1<F(I_1, I_2)$.
Analogously, substituting $I_2=I_3$ to $F(I_1, I_3)< F(I_2, I_3)$ we have $F(I_1, I_2)=I_1<F(I_2, I_2)=I_2$.
Thus, it implies that if $I_1<I_2$, then $I_1<F(I_1, I_2)<I_2$.

\textbf{Step 2:} \emph{Every equivalence class is a closed interval.}

As we have seen, every equivalence class is a (not necessarily proper) interval. Let $I_1$ be an equivalence class with endpoints $a$ and $b$. If $a=b$, then $I_1$ is a singleton and we are done. Now we assume that $a\ne b$. Suppose that $b\not\in I_1$. Then there is an equivalence class $I_2$ containing $b$.  
However, Step 1 implies that $I_1<F(I_1,I_2)<I_2$, which is a contradiction, since $b$ is on the boundary of $I_1$ and $I_2$, so there is no space for $F(I_1, I_2)$. Thus $b\in I_1$. Similar argument shows that $a\in I_1$ and hence $I_1$ is closed.

\medskip
It is important to note that the equivalence classes can be singletons, but according to our assumption $F(\alpha,\beta)=F(\beta,\alpha)$ holds for given $\alpha,\beta\in I$, hence $\alpha\sim \beta$ and there is at least one equivalence class $I_{\alpha\beta}$ which is a proper interval that contains $\alpha$ and $\beta$.

\textbf{Step 3:} \emph{Let $I$ be a proper- and $J$ be an arbitrary interval,  then
$F(I, J)$ (reps. $F(J, I)$) contained in such an equivalence class, which is a proper interval.}

If $I$ is proper, then $F(I, J)$ has at least two elements. Hence, the equivalence class containing $F(I, J)$ is an interval which is proper.

If the intersection of $I$ and $J$ is nonempty, then the statement comes immediately from Lemma \ref{l:union}.
If the intersection is empty, then we can deduce the statement from Step 1 and Step 2.

\textbf{Step 4:} \emph{The whole interval $I$ where $F$ is defined constitutes one equivalence class.}

Let us assume that we have at least two different equivalence classes $I_1$ and $I_2$. Without loss of generality, we can assume that $I_1<I_2$ and at least one of intervals is a proper one (e.g. $I_1=I_{\alpha\beta}$).

Iterating the fact (by Step 1) that $I_1<I_2$ implies $I_1<F(I_1, I_2)<I_2$, we will get infinitely many equivalence classes that are  proper intervals. Indeed, the sequence
\[
I_{\alpha\beta},\ F(I_{\alpha\beta}, I_2), F(F(I_{\alpha\beta}, I_2), I_2),\  F(F(F(I_{\alpha\beta}, I_2), I_2), I_2)
\]
gives such equivalence classes, where $I_{\alpha\beta}$ was defined in Step 2.

Let us denote the cardinality of equivalence classes by $\kappa$ and we index the equivalence classes $I_{j}$ for $j<\kappa$.
We distinguish two cases:
\begin{enumerate}
    \item $\kappa=\aleph_0$: In this case we have countably infinitely many closed, disjoint intervals that covers the closed interval $I$. This is not possible by the following theorem of Sierpinski \cite{Si} (see also \cite[p. 173]{Ku}).
    \begin{thm*}[Sierpinski]
    Let $X$ be a compact connected Hausdorff space (i.e. continuum). If $X$ has a countable cover $\{X_i\}_{i=1}^{\infty}$ by pairwise disjoint closed subsets, then at most one of the sets $X_i$ is non-empty.
    \end{thm*}

    \noindent In our case this implies that one of the equivalence classes must be the whole interval $I$.
    \smallskip
    \item $\kappa >\aleph_0$: In this case we take $\{F(I_{\alpha\beta},I_j): j<\kappa\}$. By Step 2 and by Step 3, for all $j<\kappa$ the sets $F(I_{\alpha\beta},I_j)$ are contained in equivalence classes that are pairwise disjoint and proper intervals. So we can find uncountably many disjoint, proper intervals in $\mathbb{R}$, which is impossible.
\end{enumerate}
Thus we get that every point of $I$ is in one equivalence class, hence $F$ is symmetric on  $I$. In particular, $F$ a is quasi-arithmetic mean on the compact, proper interval $I$.

\textbf{Step 5:} \emph{If $F$ is a quasi-arithmetic mean on every compact subinterval of an arbitrary interval $I$, then it is a quasi-arithmetic mean on $I$.}

The proof is based on a standard compact exhaustion argument. The interested reader is referred to the proof of Theorem 1 in \cite[p. 287]{Aczel1989}.

This finishes the 
proof of Theorem \ref{T:bisymmetryimpliescontinuity2}. 

\section{Concluding remarks}
One of the main goal concerning the characterization of bisymmetric operations without any regularity assumption is formalized in Open Problem \ref{op1}, which ask whether bisymmetry with strict increasingness would imply continuity. In the first joint paper of the authors (\cite{BKSZ2021}) there have been proved that this is true if the operation is also symmetric and reflexive (see Theorem \ref{T:bisymmetryimpliescontinuity}). Following Acz\'el's idea in the investigation of bisymmetric, strictly increasing maps (\cite{Aczel1948}) the next step would be to verify the case where symmetry condition is not assumed.

\begin{open}
Is it true or not that every bisymmetric, partially
strictly increasing, reflexive map is automatically continuous?
\end{open}

At this moment we do not know the exact answer to this question, although we believe that it is true. In this direction our present investigation would be an initial step by showing the dichotomy of symmetry of bisymmetric, strictly increasing, reflexive operations.

Furthermore, it is important to note that reflexivity in the proof of Theorem \ref{T:bisymmetryimpliescontinuity2} have not been used centrally, just implicitly in Lemma \ref{l1}. This observation leads to the following open question.

\begin{open}
Is it true or not that every bisymmetric, partially
strictly increasing, symmetric map is automatically continuous?
\end{open}

If the answer is affirmative, then with its proof we may get the analogue of Lemma \ref{l1} without reflexivity. 
Moreover, it would automatically imply the analogues of Theorem \ref{T:bisymmetryimpliescontinuity2} and the dichotomy result Corollary \ref{cor1} without the assumption of reflexivity.

\section*{Acknowledgement}
Concerning the first author the research reported in this paper is part of project no. BME-NVA-02, implemented with the support provided by the Ministry of Innovation and Technology of Hungary from the National Research, Development and Innovation Fund, financed under the TKP2021 funding scheme.

The second author was supported by Premium Postdoctoral Fellowship of the Hungarian Academy of Sciences and by the Hungarian National	Foundation for Scientific Research, Grant No. K124749.

The third author was supported by the Hungarian Academy of Sciences.

\end{document}